\documentclass[10pt]{amsart}

\usepackage{amsfonts}
\usepackage{amsthm}
\usepackage{amsmath}
\usepackage{graphicx}
\usepackage{amscd,amssymb,amsthm}

\newcounter{minutes}
\divide\time by 60
\newcounter{hours}
\multiply\time by 60 \addtocounter{minutes}{-\time}

\newcommand{\ls}{\mathbf{L}}
\newcommand{\ms}{\mathbf{M}}
\newcommand{\dt}{{\rm d}t}
\newcommand{\ds}{{\rm d}s}

\title[Functional inequalities for modified Struve functions]{Functional inequalities involving modified Struve functions}

\author[\'Arp\'ad Baricz]{\'Arp\'ad Baricz}
\address{Department of Economics, Babe\c{s}-Bolyai University, Cluj-Napoca 400591, Romania} \email{bariczocsi@yahoo.com}

\author{Tibor K. Pog\'any}
\address{Faculty of Maritime Studies, University of Rijeka, Rijeka 51000, Croatia}
\email{poganj@brod.pfri.hr}



\newtheorem{theorem}{Theorem}

\newtheorem{remark}{Remark}

\begin{document}

\def\thefootnote{}
\footnotetext{ \texttt{File:~\jobname .tex,
          printed: \number\year-0\number\month-\number\day,
          \thehours.\ifnum\theminutes<10{0}\fi\theminutes}
} \makeatletter\def\thefootnote{\@arabic\c@footnote}\makeatother

\maketitle
\allowdisplaybreaks

\begin{center}
{\em Dedicated to Professor Paul L. Butzer on the occasion of his 85th birthday}
\end{center}

\begin{abstract}
In this paper our aim is to prove some monotonicity and convexity results for the modified Struve function of the second kind by using its integral representation. Moreover, as consequences of these results, we present some functional inequalities (like Tur\'an type inequalities) as well as lower and upper bounds for modified Struve function of the second kind and its logarithmic derivative.
\end{abstract}

\section{\bf Introduction}
\setcounter{equation}{0}

In the last decades many functional inequalities and monotonicity properties for special functions (like Bessel, modified Bessel, Gaussian hypergeometric, Kummer hypergeometric) and their combinations have been established by researchers, motivated by several problems that arise in wave mechanics, fluid mechanics, electrical engineering, quantum billiards, biophysics, mathematical physics, finite elasticity, probability and statistics, special relativity and radar signal processing. Although the inequalities involving the quotients of modified Bessel functions of the first and second kind are interesting in their own right, recently the lower and upper bounds for such quotients have received increasing attention, since they play an important role in various problems of mathematical physics and electrical engineering. For more details, see for example \cite{bariczedin} and the references therein. The modified Struve functions of the first and second kind are related to modified Bessel functions of the first kind, and thus their properties can be useful in problems of mathematical physics. In \cite{joshi} Joshi and Nalwaya presented some two-sided inequalities for modified Struve functions of the first kind and for their ratios. They also deduced some Tur\'an and Wronski type inequalities for modified Struve functions of the first kind by using a generalized hypergeometric function representation of the Cauchy product of two modified Struve functions of the first kind. Motivated by the above results, by using a known result on the monotonicity of quotients of MacLaurin series, recently in \cite{bp} we proved some monotonicity and convexity results for the modified Struve functions of the first kind. Moreover, as consequences of these results, we presented some functional inequalities as well as lower and upper bounds for modified Struve functions of the first kind. In this paper our aim is to continue the study from \cite{bp}, but for the modified Struve functions of the second kind. The key tools in the proofs of the main results are the techniques developed in the extensive study of modified Bessel functions of the first and second kind and their ratios. The difficulty in the study of the modified Struve function consists in the fact that the modified Struve differential equation is not homogeneous, however, as we can see below, the integral representation of modified Struve function of the second kind is very useful in order to study its monotonicity and convexity properties.

\section{\bf Modified Struve function: Monotonicity patterns and functional inequalities}
\setcounter{equation}{0}

The modified Struve functions of the first and second kind, $\ls_{\nu}$ and $\ms_{\nu}$ are particular solutions of the modified Struve equation \cite[p. 288]{nist}
   \begin{equation} \label{eqStruve}
      x^2y''(x)+xy'(x) - (x^2+\nu^2)y(x) = \frac{x^{\nu+1}}{\sqrt{\pi}2^{\nu-1}\Gamma\left(\nu+\frac{1}{2}\right)}.
   \end{equation}
The modified Struve function of the second kind has the power series representation
   \begin{equation} \label{L1}
      \ms_{\nu}(x) = \ls_{\nu}(x)-I_{\nu}(x) = \sum_{n = 0}^\infty \frac{\left(\frac{x}{2}\right)^{2n+\nu+1}}
                     {\Gamma\left(n+\frac{3}{2}\right)\Gamma\left(n+\nu+\frac{3}{2}\right)}
                   - \sum_{n = 0}^\infty \frac{\left(\frac{x}{2}\right)^{2n+\nu}}
                     {\Gamma\left(n+1\right)\Gamma\left(n+\nu+1\right)},
   \end{equation}
where $I_{\nu}$ stands for the modified Bessel function of the first kind.

Now, consider the function $\mathcal{M}_{\nu}:(0,\infty)\to\mathbb{R},$ defined by
   \begin{equation}\label{defm}
      \mathcal{M}_{\nu}(x)=-2^{\nu}\Gamma\left(\nu+\frac{1}{2}\right)x^{-\nu}\ms_{\nu}(x),
   \end{equation}
which for $\nu>-\frac{1}{2}$ has the integral representation \cite[p. 292]{nist}
   \begin{equation}\label{integr}
      \mathcal{M}_{\nu}(x)=\frac{2}{\sqrt{\pi}}\int_0^1\left(1-t^2\right)^{\nu-\frac{1}{2}}e^{-xt}\dt
   \end{equation}

Our main result is the following theorem.

\begin{theorem}\label{th1}
The following assertions are true:
\begin{enumerate}
   \item[\bf a.] The function $x\mapsto \mathcal{M}_{\nu}(x)$ is completely monotonic and log-convex on $(0,\infty)$ for
                 all   $\nu>-\frac{1}{2}.$
   \item[\bf b.] The function $\nu\mapsto \mathcal{M}_{\nu}(x)$ is completely monotonic and log-convex on
                 $\left(-\frac{1}{2},\infty\right)$ for all $x>0.$
   \item[\bf c.] The function $x\mapsto -\ms_{\nu}(x)$ is completely monotonic and log-convex on $(0,\infty)$ for all
                 $\nu\in\left[-\frac{1}{2},0\right].$
\end{enumerate}
Furthermore, for all $x>0$ the following inequalities are valid:
   \begin{equation}\label{bound0}
      \mathcal{M}_{\nu}(x)< \frac{\Gamma\left(\nu+\frac{1}{2}\right)}{\Gamma(\nu+1)},\ \ \ \nu>-\frac{1}{2},
   \end{equation}
   \begin{equation}\label{ineqturan}
      0<\left[\ms_{\nu}(x)\right]^2-\ms_{\nu-1}(x)\ms_{\nu+1}(x) < \frac{\left[\ms_{\nu}(x)\right]^2}{\nu+\frac{1}{2}},\ \ \
       \nu>\frac{1}{2},
   \end{equation}
   \begin{equation}\label{quot1}
      \frac{x\ms_{\nu}'(x)}{\ms_{\nu}(x)}<\nu, \ \ \ \nu>-\frac{1}{2},
   \end{equation}
   \begin{equation}\label{quot2}
      - \sqrt{x^2+\nu^2}<\frac{x\ms_{\nu}'(x)}{\ms_{\nu}(x)}<\sqrt{x^2+\nu^2}, \ \ \ \nu>\frac{1}{2},
   \end{equation}
\end{theorem}

\begin{proof}[\bf Proof]
{\bf a.} \& {\bf b.} By \eqref{defm} for $n,m\in\{0,1,\dots\}$ and $\nu>-\frac{1}{2}$ we have
   \[ (-1)^{n}\left[\mathcal{M}_{\nu}(x)\right]^{(n)} = \frac2{\sqrt{\pi}}\int_0^1t^n\left(1-t^2\right)^{\nu-\frac{1}{2}}e^{-xt}\dt,\]
   \[ (-1)^{m}\frac{\partial^m\mathcal{M}_{\nu}(x)}{\partial\nu^m} = \frac2{\sqrt{\pi}}\int_0^1
      \left(\log\frac{1}{1-t^2}\right)^m\left(1-t^2\right)^{\nu-\frac{1}{2}}e^{-xt}\dt.\]
Thus, the functions $x\mapsto \mathcal{M}_{\nu}(x)$ and $\nu\mapsto \mathcal{M}_{\nu}(x)$ are indeed completely monotonic and consequently are log-convex, since every completely monotonic function is log-convex, see \cite[p. 167]{widder}. Alternatively, the log-convexity of these functions can be proved also by using \eqref{integr} and the H\"older-Rogers inequality for integrals.

For inequality \eqref{bound0} just observe that $\mathcal{M}_{\nu}$ is decreasing on $(0,\infty)$ for all $\nu>-\frac{1}{2},$ and thus
   \[ \mathcal{M}_{\nu}(x) < \frac2{\sqrt{\pi}}\int_0^1\left(1-t^2\right)^{\nu-\frac{1}{2}}\dt
                           = \frac1{\sqrt{\pi}}\int_0^1s^{-\frac{1}{2}}(1-s)^{\nu-\frac{1}{2}}\ds
                           = \frac{\Gamma(\nu + \frac12)}{\Gamma(\nu+1)}.\]

{\bf c.} By \eqref{defm}
   \[ -\ms_{\nu}(x)=\frac{x^{\nu}\mathcal{M}_{\nu}(x)}{2^{\nu}\Gamma\left(\nu+\frac{1}{2}\right)}.\]
On the other hand, observe that $x\mapsto x^{\nu}$ is completely monotonic on $(0,\infty)$ for all $\nu\leq0.$ Thus, by part {\bf a} of the theorem the function $x\mapsto -\ms_{\nu}(x),$ as a product of two completely monotonic functions, is completely monotonic and log-convex on $(0,\infty)$ for all $\nu\in\left(-\frac{1}{2},0\right].$ Now, since (see \cite[p. 254]{nist} and \cite[p. 291]{nist})
   \[ \ms_{-\frac{1}{2}}(x)=\ls_{-\frac{1}{2}}(x)-I_{-\frac{1}{2}}(x) = \sqrt{\frac{2}{\pi x}}\sinh x
           -\sqrt{\frac{2}{\pi x}}\cosh x=-\sqrt{\frac{2}{\pi x}}\, e^{-x},\]
the function $x\mapsto -\ms_{-\frac{1}{2}}(x)$ is completely monotonic and log-convex as the product of the completely monotonic and log-convex functions $x\mapsto \sqrt{\frac{2}{\pi}}x^{-\frac{1}{2}}$ and $x\mapsto e^{-x}.$

Now, focus on the Tur\'an type inequality \eqref{ineqturan}. Since $\nu\mapsto \mathcal{M}_{\nu}(x)$ is log-convex on $\left(-\frac{1}{2},\infty\right)$ for $x>0,$ it follows that for all $\nu_1,\nu_2>-\frac{1}{2},$ $\alpha\in[0,1]$ and $x>0$ we have
   \[ \mathcal{M}_{\alpha\nu_1+(1-\alpha)\nu_2}(x) \leq
        \left[\mathcal{M}_{\nu_1}(x)\right]^{\alpha}\left[\mathcal{M}_{\nu_2}(x)\right]^{1-\alpha}.\]
Choosing $\nu_1= \nu-1,$ $\nu_2 = \nu+1$ and $\alpha=\frac{1}{2},$ the above inequality reduces to the Tur\'an type inequality
   \[ \left[\mathcal{M}_{\nu}(x)\right]^2 - \mathcal{M}_{\nu-1}(x)\mathcal{M}_{\nu+1}(x) \leq 0,\]
which by \eqref{defm} is equivalent to the right-hand side of \eqref{ineqturan}. For the left-hand side \eqref{ineqturan} observe
that the Tur\'anian
   \[ {}_{\ms}\Delta_{\nu}(x) = \left[\ms_{\nu}(x)\right]^2-\ms_{\nu-1}(x)\ms_{\nu+1}(x)\]
can be rewritten as
   \begin{equation}\label{delta}
      {}_{\ms}\Delta_{\nu}(x) = {}_{I}\Delta_{\nu}(x)+{}_{\ls}\Delta_{\nu}(x)+{}_{I,\ls}\Delta_{\nu}(x),
   \end{equation}
where
   \[ {}_{I}\Delta_{\nu}(x) = \left[I_{\nu}(x)\right]^2-I_{\nu-1}(x)I_{\nu+1}(x),\]
   \[ {}_{\ls}\Delta_{\nu}(x) = \left[\ls_{\nu}(x)\right]^2-\ls_{\nu-1}(x)\ls_{\nu+1}(x)\]
and
   \[ {}_{I,\ls}\Delta_{\nu}(x) = I_{\nu+1}(x)\ls_{\nu-1}(x) + I_{\nu-1}(x)\ls_{\nu+1}(x) - 2I_{\nu}(x)\ls_{\nu}(x).\]
It is well-known (see \cite{bariczbams, bp, joshi, thiru}) that ${}_I\Delta_{\nu}(x)>0$ for all $\nu>-1$ and $x>0,$
and ${}_{\ls}\Delta_{\nu}(x)>0$ for all $\nu>-\frac{3}{2}$ and $x>0.$ On the other hand, by using the integral representations
   \[ I_{\nu}(x) = \frac{2\left( \frac 12 x\right)^\nu}{\sqrt{\pi} \Gamma\left(\nu+\frac12\right)}
                   \int_0^1 (1-t^2)^{\nu-\frac12} \cosh(xt)\dt, \]
   \[ \ls_{\nu}(x) = \frac{2\left( \frac 12 x\right)^\nu}{\sqrt{\pi} \Gamma\left(\nu+\frac12\right)}
                     \int_0^1 (1-t^2)^{\nu-\frac12} \sinh(xt)\dt, \]
where $\nu>-\frac{1}{2},$ and the relation $\Gamma\left(\nu+\frac32\right)\Gamma\left(\nu-\frac12\right) = \Gamma^2\left(\nu+\frac12\right)$ we obtain for all $\nu>\frac{1}{2}$ and $x>0$ that
   \begin{align*}
      {}_{I,\ls}\Delta_{\nu}(x) &=\frac{4\left(\frac 12 x\right)^{2\nu}}
                 {\pi\Gamma\left(\nu+\frac32\right)\Gamma\left(\nu - \frac12\right)}\left[
                 \int_0^1\int_0^1(1-t^2)^{\nu+\frac12}(1-s^2)^{\nu-\frac32} \cosh(xt) \sinh(xs)\dt\ds \right. \\
       &\qquad + \left. \int_0^1\int_0^1(1-t^2)^{\nu-\frac32}(1-s^2)^{\nu+\frac12} \cosh(xt) \sinh(xs)\dt\ds \right] \\
       &\qquad - \frac{8\left(\frac 12 x\right)^{2\nu}}{\pi \Gamma^2\left(\nu+\frac12\right)} \int_0^1\int_0^1
                 (1-t^2)^{\nu-\frac12}(1-s^2)^{\nu-\frac12} \cosh(xt) \sinh(xs)\dt\ds \\
       &= \frac{4\left(\frac 12 x\right)^{2\nu}}{\pi \Gamma^2\left(\nu+\frac12\right)}\int_0^1\int_0^1(1-t^2)^{\nu-\frac32}
                 (1-s^2)^{\nu-\frac32}\left[(1-t^2)^2 + (1-s^2)^2 \right.\\
       &\qquad \left. - 2(1-t^2)(1-s^2)\right] \cosh(xt) \sinh(xs)\dt\ds\\
       &= \frac{4\left(\frac 12 x\right)^{2\nu}}{\pi\Gamma^2\left(\nu+\frac12\right)} \int_0^1\int_0^1(1-t^2)^{\nu-\frac32}
                 (1-s^2)^{\nu-\frac32}(t^2-s^2)^2 \cosh(xt) \sinh(xs)\dt\ds .
   \end{align*}
This shows that ${}_{I,\ls}\Delta_{\nu}(x)>0$ for all $\nu>\frac{1}{2}$ and $x>0,$ and consequently by \eqref{delta} is    ${}_{\ms}\Delta_{\nu}(x)>0$ for all $\nu>\frac{1}{2}$ and $x>0$.

Next we prove inequalities \eqref{quot1} and \eqref{quot2}. Since for $\nu>-\frac{1}{2}$ the function $\mathcal{M}_{\nu}$ is completely monotonic on $(0,\infty),$ it follows that it is decreasing on $(0,\infty)$ for $\nu>-\frac{1}{2}.$ Then, by \eqref{defm}
the function $x\mapsto \log\left(-x^{-\nu}\ms_{\nu}(x)\right)$ is also decreasing on $(0,\infty)$ for $\nu>-\frac{1}{2},$ which in turn implies inequality \eqref{quot1}. Now, we show that inequalities \eqref{ineqturan} and \eqref{quot1} imply inequality \eqref{quot2}. For this first observe that if we use the recurrence relations (see \cite[p. 251]{nist} and \cite[p. 292]{nist}) for the functions $\ls_{\nu}$ and $I_{\nu}$ and \eqref{L1} it can be shown that the function $\ms_{\nu}$ satisfies the same recurrence relations as $\ls_{\nu},$ that is,
   \begin{equation} \label{rec0}
      \ms_{\nu-1}(x)-\ms_{\nu+1}(x) = \frac{2\nu}{x}\ms_{\nu}(x)
                                    + \frac{\left(\frac{x}{2}\right)^{\nu}}{\sqrt{\pi}\,\Gamma\left(\nu+\frac{3}{2}\right)},
   \end{equation}
   \begin{equation} \label{L2}
      \ms_{\nu-1}(x) + \ms_{\nu+1}(x) =
                       2\ms_{\nu}'(x)-\frac{\left(\frac{x}{2}\right)^{\nu}}{\sqrt{\pi}\,\Gamma\left(\nu+\frac{3}{2}\right)},
   \end{equation}
   \begin{equation} \label{rec1}
      x\ms_{\nu}'(x)+\nu\ms_{\nu}(x)=x\ms_{\nu-1}(x).
   \end{equation}
Subtracting the recurrence relations \eqref{rec0} and \eqref{L2} we obtain
   \begin{equation}\label{rec2}
      \ms_{\nu+1}(x) = \ms_{\nu}'(x)-\frac{\nu}{x}\ms_{\nu}(x) - \frac{\left(\frac{x}{2}\right)^{\nu}}
                       {\sqrt{\pi}\, \Gamma\left(\nu+\frac{3}{2}\right)}.
   \end{equation}
From \eqref{rec1} and \eqref{rec2} it follows that
   \begin{equation} \label{L2.5}
      {}_{\ms}\Delta_{\nu}(x) = \left(1+\frac{\nu^2}{x^2}\right)\left[\ms_{\nu}(x)\right]^2 - \left[\ms_{\nu}'(x)\right]^2
                              + \frac{x^{\nu}\ms_{\nu-1}(x)}{\sqrt{\pi}2^{\nu}\Gamma\left(\nu+\frac{3}{2}\right)}.
   \end{equation}
But, according to the left-hand side of \eqref{ineqturan} we have ${}_{\ms}\Delta_{\nu}(x)>0$ for $x>0$ and $\nu>\frac{1}{2},$ and consequently
   \[ \left(1+\frac{\nu^2}{x^2}\right)\left[\ms_{\nu}(x)\right]^2-\left[\ms_{\nu}'(x)\right]^2 >
         - \frac{x^{\nu}\ms_{\nu-1}(x)}{\sqrt{\pi}2^{\nu}\Gamma\left(\nu+\frac{3}{2}\right)}>0.\]
Therefore, for $x>0$ and $\nu>\frac{1}{2}$ we have
   \[ \left(\frac{x\ms_{\nu}'(x)}{\ms_{\nu}(x)} - \sqrt{x^2+\nu^2}\right)\left(\frac{x\ms_{\nu}'(x)}
                {\ms_{\nu}(x)}+\sqrt{x^2+\nu^2}\right)<0.\]
Inequality \eqref{quot1} implies the right-hand side of \eqref{quot2}, while the above inequality imply the left-hand side of \eqref{quot2}.
\end{proof}

\section{\bf Further results}
\setcounter{equation}{0}

In this section we give a set of other functional inequalities for the function $x \mapsto \mathcal M_\nu(x)$.

\begin{theorem} \label{th2}
The following inequalities hold true:
   \begin{enumerate}
      \item[\bf a.] For all $x,y>0$ and $\nu>-\frac{1}{2}$ we have
         \begin{equation} \label{FX1}
            \mathcal{M}_{\nu}(x+y) \geq \frac{\Gamma(\nu+1)}{\Gamma\left(\nu+\frac{1}{2}\right)}
                                        \, \mathcal{M}_{\nu}(x)\mathcal{M}_{\nu}(y)\, .
         \end{equation}
      \item[\bf b.] For all $\nu\geq \frac{1}{2}$ and $x>0$ it is
         \begin{equation} \label{bound1}
            \mathcal{M}_{\nu}(x)\geq\frac{\Gamma\left(\nu+\frac{1}{2}\right)}{\Gamma(\nu+1)} \cdot \frac{1-e^{-x}}{x}\, .
         \end{equation}
        Moreover, the above inequality is reversed when $|\nu|<\frac{1}{2}$ and $x>0.$
      \item[\bf c.] For all $\nu\geq \frac{3}{2}$ and $x>0,$ we have
         \begin{equation} \label{FX2}
            \mathcal{M}_{\nu-1}(x)\mathcal{M}_{\nu+1}(x) \leq \mathcal{M}_{\frac{1}{2}}(x)\mathcal{M}_{2\nu-\frac{1}{2}}(x)\, ,
         \end{equation}
        which is reversed when $\nu\in\left(\frac{1}{2},\frac{3}{2}\right)$.
      \item[\bf d.] For all for all $\nu>-1$ and $x>0$, we have
         \begin{equation} \label{FX3}
            \mathcal{M}_{\nu}(x) < \frac{\Gamma\left(\nu+\frac{1}{2}\right)}{\Gamma(\nu+1)}e^{\frac{x^2}{4(\nu+1)}}
                         - \frac{4}{\sqrt{\pi}(2\nu+1)}\sinh \frac{x}{2\nu+3}\, .
         \end{equation}
   \end{enumerate}
\end{theorem}

\begin{proof}[\bf Proof.]
{\bf a.} From \eqref{bound0} and part {\bf a} of Theorem 1 it is clear that the function $x \mapsto \frac{\Gamma(\nu+1)}{\Gamma\left(\nu+\frac{1}{2}\right)}\mathcal{M}_{\nu}(x)$ maps $(0,\infty)$ into $(0,1)$ and it is completely monotonic on $(0,\infty)$ for all $\nu>-\frac{1}{2}.$ On the other hand, according to Kimberling \cite{kimber} if a function $f,$ defined on $(0,\infty),$ is continuous and completely monotonic and maps $(0,\infty)$ into $(0,1),$ then $\log f$ is super-additive, that is for all $x,y > 0$ we have
   \[ \log f(x + y) \geq \log f(x) + \log f(y)\quad \mbox{or}\quad  f(x + y)\geq f(x)f(y).\]
Therefore we conclude the asserted inequality \eqref{FX1}.

{\bf b.} We point out that \eqref{bound1} complements and improves inequality \eqref{bound0}. Moreover, because
\eqref{bound1} inequality is reversed when $|\nu|<\frac{1}{2}$ and $x>0,$ and since  $e^{-x}>1-x,$ the reversed form of inequality \eqref{bound1} is better than \eqref{bound0} for $|\nu|<\frac{1}{2}$ and $x>0.$ Now, recall the Chebyshev integral inequality \cite[p. 40]{mitri}: {\it If $f,g:[a,b]\rightarrow\mathbb{R}$ are synchoronous (both increase or decrease) integrable functions, and $p:[a,b]\rightarrow\mathbb{R}$ is a positive integrable function, then
   \begin{equation}\label{csebisev}
      \int_a^bp(t)f(t)\dt \int_a^bp(t)g(t)\dt \leq \int_a^bp(t)\dt \int_a^bp(t)f(t)g(t)\dt.
   \end{equation}
Note that if $f$ and $g$ are asynchronous (one is decreasing and the other is increasing), then \eqref{csebisev} is reversed}. Now, we shall use \eqref{csebisev} and \eqref{integr} to prove \eqref{bound1}. For this consider the functions $p,f,g:[0,1]\to\mathbb{R},$ defined by $$p(t)=1, \ f(t)=\frac{2}{\sqrt{\pi}}(1-t^2)^{\nu-\frac{1}{2}}\ \ \mbox{and} \ \ g(t)=e^{-xt}.$$ Observe that $g$ is decreasing and $f$ is increasing (decreasing) if $-\frac{1}{2}<\nu\leq\frac{1}{2}$ ($\nu\geq\frac{1}{2}$). On the other hand, we have
   \[ \mathcal{M}_{\nu}(0) = \frac2{\sqrt{\pi}} \int_0^1\left(1-t^2\right)^{\nu-\frac{1}{2}}\dt
                           = \frac{\Gamma\left(\nu+\frac{1}{2}\right)}{\Gamma(\nu+1)}\quad \mbox{and}\quad
                             \int_0^1e^{-xt}\dt=\frac{1-e^{-x}}{x},\]
and by the Chebyshev inequality \eqref{csebisev} we get inequality \eqref{bound1} when $\nu\geq\frac{1}{2},$ and its reverse when $|\nu|<\frac{1}{2}.$

{\bf c.} Another use of the Chebyshev integral inequality \eqref{csebisev}, that is $p,f,g:[0,1] \to \mathbb{R},$ defined by
   \[ p(t) = e^{-xt},\quad f(t) = \frac2{\sqrt{\pi}}\,(1-t^2)^{\nu-\frac{3}{2}}\quad \mbox{and}\quad
       g(t)= \frac2{\sqrt{\pi}}\,(1-t^2)^{\nu+\frac{1}{2}},\]
taking into account (see \cite[p. 254]{nist} and \cite[p. 291]{nist})
   \[ \ms_{\frac{1}{2}}(x) = \ls_{\frac{1}{2}}(x) - I_{\frac{1}{2}}(x) = \sqrt{\frac2{\pi x}}\,(\cosh x -1)
                           - \sqrt{\frac2{\pi x}}\,\sinh x  = \sqrt{\frac2{\pi x}}\,(e^{-x}-1),\]
by \eqref{defm} results in \eqref{FX2} for $\nu\geq \frac{3}{2}$ and $x>0.$ In turn, the above inequality is reversed when $\nu\in\left(\frac{1}{2},\frac{3}{2}\right)$ and $x>0.$

{\bf d.} If we combine the inequalities \cite{bariczedin,bp}
   \[ I_{\nu}(x)   < \frac{x^{\nu}}{2^{\nu}\Gamma(\nu+1)}e^{\frac{x^2}{4(\nu+1)}}\quad \mbox{and}\quad
      \ls_{\nu}(x) > \frac{x^{\nu}\sinh \frac{x}{2\nu+3}}{\sqrt{\pi}2^{\nu-1}\Gamma\left(\nu+\frac{3}{2}\right)},\]
which hold for all $\nu>-1$ and $x>0$ by \eqref{L1} we obtain
   \[ \ms_{\nu}(x) > \frac{x^{\nu}\sinh \frac{x}{2\nu+3}}{\sqrt{\pi}2^{\nu-1}\Gamma\left(\nu+\frac{3}{2}\right)} -
                           \frac{x^{\nu}}{2^{\nu}\Gamma(\nu+1)}e^{\frac{x^2}{4(\nu+1)}}, \]
consequently \eqref{FX3} as well.
\end{proof}

\begin{remark} {\em By \eqref{defm} and \eqref{FX2} immediately follows the inequality
   \[ \ms_{\nu-1}(x)\ms_{\nu+1}(x) \leq \frac{\sqrt{2}\, \Gamma(2\nu)(e^{-x}-1)}
                   {\sqrt{\pi x}\,\Gamma\left(\nu-\frac{1}{2}\right)\Gamma\left(\nu+\frac{3}{2}\right)}
                   \ms_{2\nu-\frac{1}{2}}(x); \]
the validity range is $\nu\geq \frac{3}{2}$ and $x>0,$ while the inequality is reversed when
$\nu\in\left(\frac{1}{2},\frac{3}{2}\right)$ and $x>0.$}
\end{remark}

Next we derive inequalities similar to \eqref{quot2} when $\nu\in\left[-\frac{1}{2},0\right].$ We note that the left-hand side of \eqref{quot3} is weaker than the left-hand side of \eqref{quot2}, however, the right-hand side of \eqref{quot3} is better than the right-hand side of \eqref{quot2}

\begin{theorem} \label{th3}
For all $x>0$ and $\nu\in\left[-\frac{1}{2},0\right]$ we have
   \begin{equation} \label{quot3}
      \frac{-1-\sqrt{1+4(x^2+\nu^2)}}{2} < \frac{x\ms_{\nu}'(x)}{\ms_{\nu}(x)} < \frac{-1+\sqrt{1+4(x^2+\nu^2)}}{2}\, .
   \end{equation}
Moreover, for $x>0$ and $\nu>\frac{1}{2}$ we have
   \begin{equation} \label{FX31}
      \left[\frac{x\ms_{\nu}'(x)}{\ms_{\nu}(x)}\right]'<\frac{x}{\nu+\frac{1}{2}}.
   \end{equation}
\end{theorem}

\begin{proof}[\bf Proof.]
By {\bf c} of Theorem 1 we have
   \[ \ms_{\nu}''(x)\ms_{\nu}(x) - \left[\ms_{\nu}'(x)\right]^2 > 0 \]
for all $x>0$ and $\nu\in\left[-\frac{1}{2},0\right].$ On the other hand, recall that the modified Struve function $\ms_{\nu}$ is a particular solution of the modified Struve equation \eqref{eqStruve} and consequently
   \begin{equation} \label{L4}
      \ms_{\nu}''(x) = \left(1+\frac{\nu^2}{x^2}\right)\ms_{\nu}(x) -
                       \frac{1}{x}\ms_{\nu}'(x)+\frac{x^{\nu-1}}{\sqrt{\pi}2^{\nu-1}\Gamma\left(\nu+\frac{1}{2}\right)}.
   \end{equation}
Combining this equation with the above inequality we get
   \[ \left(1 + \frac{\nu^2}{x^2}\right)\left[\ms_{\nu}(x)\right]^2 - \frac{1}{x}\ms_{\nu}(x)\ms_{\nu}'(x)
              - \left[\ms_{\nu}'(x)\right]^2>0,\]
that is,
   \[ \left[\frac{x\ms_{\nu}'(x)}{\ms_{\nu}(x)}\right]^2+\frac{x\ms_{\nu}'(x)}{\ms_{\nu}(x)}-(x^2+\nu^2)<0.\]
Here we used the fact that $\ms_{\nu}(x)<0$ for $x>0$ and $\nu\geq-\frac{1}{2}.$ From the above inequality we deduce \eqref{quot3},
for all $x>0$ and $\nu\in\left[-\frac{1}{2},0\right].$ Moreover, since $\ms_{\nu}'(x)>0$ for $x>0$ and $\nu\in\left[-\frac{1}{2},0\right],$ the expression
$x\ms_{\nu}'(x)\,\big[ \ms_{\nu}(x) \big]^{-1}$ is negative, which implies the right-hand side of \eqref{quot3}. \medskip

It remains to prove \eqref{FX31}. By using \eqref{rec1} and \eqref{L2.5} we have
   \[ \frac{1}{x}\left[\ms_{\nu}(x)\right]^2\left[\frac{x\ms_{\nu}'(x)}{\ms_{\nu}(x)}\right]' =
                 \left(1+\frac{\nu^2}{x^2}\right)\left[\ms_{\nu}(x)\right]^2-\left[\ms_{\nu}'(x)\right]^2
               + \frac{\left(\nu+\frac{1}{2}\right)x^{\nu-1}\ms_{\nu}(x)}{\sqrt{\pi}2^{\nu-1}\Gamma\left(\nu+\frac{3}{2}\right)}.\]
Thus, by using \eqref{L2.5}, \eqref{rec1}, \eqref{quot1} and the fact that $\ms_{\nu}(x)<0$ for $x>0$ and $\nu>-\frac{1}{2},$ we have
   \[ {}_{\ms}\Delta_{\nu}(x)-\frac{1}{x}\left[\ms_{\nu}(x)\right]^2\left[\frac{x\ms_{\nu}'(x)}{\ms_{\nu}(x)}\right]' =
              \frac{x^{\nu}\ms_{\nu}(x)}{2^{\nu}\sqrt{\pi}\Gamma\left(\nu+\frac{3}{2}\right)}\left[\frac{\ms_{\nu}'(x)}
              {\ms_{\nu}(x)}-\frac{\nu+1}{x}\right]>0.\]
Combining this with the right-hand side of the Tur\'an type inequality \eqref{ineqturan} we obtain the desired bound. \medskip
\end{proof}

Now, in order to establish a bilateral functional inequality for $\mathcal M_\nu$, we need the Fox-Wright generalized hypergeometric
function ${}_p\Psi_q(\cdot)$, with $p$ numerator and $q$ denominator parameters, defined by
   \begin{equation} \label{W1}
      {}_p\Psi_q \left[\left. \begin{array}{c} (a_1,\alpha_1), \dots, (a_p,\alpha_p) \\
                 (b_1,\beta_1),\dots, (b_q,\beta_q) \end{array} \right| z \right]
               = \sum_{n=0}^\infty \dfrac{\prod_{l=1}^p\Gamma(a_l+\alpha_l n)}
                 {\prod_{j=1}^q\Gamma(b_j+\beta_j n)} \dfrac{z^n}{n!}.
   \end{equation}
Here $z,a_l,\,b_j\in \mathbb{C}$, $\alpha_l,\,\beta_j \in \mathbb{R}$ for $l\in\{1,\dots,p\}$ and $j\in\{1,\dots,q\}$. The series \eqref{W1} converges absolutely and uniformly for all bounded $|z|,\, z \in \mathbb C$ when
   \[ \varepsilon= 1 + \sum_{j=1}^q\beta_j - \sum_{l=1}^p\alpha_l>0\,.\]
We note that the next inequality complements and improve \eqref{bound0}.

\begin{theorem} \label{Th4}
For all $\nu > -\tfrac12$ and $x>0$ we have
   \[ \frac{\Gamma\left(\nu+\frac{1}{2}\right)}{\Gamma(\nu+1)}\,
            e^{-\frac{\Gamma(\nu+1)\,x}{\sqrt{\pi}\,\Gamma\left(\nu+\frac32\right)}}
      \leq  \mathcal{M}_{\nu}(x) \leq \frac{\Gamma\left(\nu+\frac{1}{2}\right)}{\Gamma(\nu+1)}
      -     \frac{1-e^{-x}}{\sqrt{\pi}\,(\nu+\frac12)}.\]
\end{theorem}

\begin{proof}[\bf Proof]
By \eqref{integr} we have
   \begin{align*}
      \dfrac{\sqrt{\pi}}2\,\mathcal M_\nu(x) &= \int_0^1 (1-t^2)^{\nu-\frac12}\, {e}^{-xt}\, {\rm d}t
                  = \dfrac12\, \sum_{n=0}^\infty \dfrac{(-x)^n}{n!}\, \int_0^1 s^{\frac{n-1}2}(1-s)^{\nu - \frac12}\, {\rm d}s \\
                 &= \dfrac{\Gamma(\nu+\frac12)}2\,\sum_{n=0}^\infty \dfrac{\Gamma\left( \frac{n+1}2\right)}
                    {\Gamma\left(\frac n2+\nu+1\right)}\,\dfrac{(-x)^n}{n!} = \dfrac{\Gamma(\nu+\frac12)}2\,
                    {}_1\Psi_1\left[ \left.\begin{array}{c} (\frac12, \frac12)\\
                    \left(\nu+1, \frac12\right) \end{array} \right| -x\right]\, .
   \end{align*}
Since $\varepsilon = 1$, the series converges for all $x >0$. Therefore, for all $x >0$ we have
   \[ \mathcal M_\nu(x) = \dfrac{\Gamma(\nu+\frac12)}{\sqrt{\pi}}\,
               {}_1\Psi_1\left[ \left.\begin{array}{c} (\frac12, \frac12)\\
               \left(\nu+1, \frac12\right) \end{array} \right| -x\right].\]
On the other hand, recall \cite[Theorem 4]{PS} and \cite[eq. (22)]{PS}, which say that for all ${}_p\Psi_q[\cdot]$ satisfying
   \begin{equation} \label{FX4}
       \psi_1> \psi_2\ \ \  \mbox{and}\ \  \ \psi_1^2 < \psi_2\psi_0,
   \end{equation}
the two--sided inequality
   \begin{equation}\label{ineqPS}
      \psi_0e^{\psi_1\psi_0^{-1}|x|}\leq {}_p\Psi_q \left[\left. \begin{array}{c} (a_1,\alpha_1), \dots, (a_p,\alpha_p) \\
      (b_1,\beta_1),\dots, (b_q,\beta_q) \end{array} \right| x \right]\leq \psi_0-(1-e^{|x|})\psi_1,
   \end{equation}
hold for all $x \in \mathbb R$. Here
   \[ \psi_m = \frac{\prod_{j=1}^p\Gamma(a_j+\alpha_jm)}{\prod_{j=1}^q\Gamma(b_j+\beta_jm)}, \qquad j\in\{0,1,2\}.\]
In our case we have
   \[ \psi_0e^{\psi_1\psi_0^{-1}|x|}=\frac{\sqrt{\pi}}{\Gamma(\nu+1)}e^{\frac{\Gamma(\nu+1)\,|x|}{\sqrt{\pi}\Gamma\left(\nu+\frac32\right)}}
               \quad \mbox{and}\quad \psi_0-(1-e^{|x|})\psi_1=\frac{\sqrt{\pi}}{\Gamma(\nu+1)} - \frac{1-e^{|x|}}{\Gamma\left(\nu+\frac32\right)},\]
and the conditions \eqref{FX4} can be simplified as
\begin{equation}\label{gammaineq} \frac{2}{\sqrt{\pi}}>\frac{\Gamma\left(\nu+\frac32\right)}{\Gamma(\nu+2)} > \sqrt{\frac2{\pi(\nu+1)}}. \end{equation}
In what follows, we show that if $\nu>-\frac{1}{2},$ then \eqref{gammaineq} holds, and consequently, by applying \eqref{ineqPS}, for all $\nu > -\frac{1}{2}$ and $x>0$ we achieve the asserted bilateral inequality.

Consider the functions $f,g:(-1,\infty),$ defined by
$$f(\nu)=\frac{\sqrt{\pi}}{2}\frac{\Gamma\left(\nu+\frac{3}{2}\right)}{\Gamma(\nu+2)}\ \ \mbox{and}\ \ g(\nu)=\sqrt{\frac{\pi}{2}}\sqrt{\nu+1}\cdot\frac{\Gamma\left(\nu+\frac{3}{2}\right)}{\Gamma(\nu+2)}.$$
Since Euler's digamma function $\psi,$ defined by $\psi(x)=\Gamma'(x)/\Gamma(x),$ is increasing on $(0,\infty),$ we obtain that
$$\frac{f'(\nu)}{f(\nu)}=\psi\left(\nu+\frac{3}{2}\right)-\psi(\nu+2)<0$$
for all $\nu>-1,$ and thus $f(\nu)<f\left(-\frac{1}{2}\right)=1$ if $\nu>-\frac{1}{2}.$ This proves the left-hand side of \eqref{gammaineq}. Now, for the right-hand side of \eqref{gammaineq} we consider the function $h:(-1,\infty),$ defined by
$$h(\nu)=\frac{g'(\nu)}{g(\nu)}=\psi\left(\nu+\frac{3}{2}\right)-\psi(\nu+2)+\frac{1}{2(\nu+1)}.$$
By using the formulas \cite[p. 140]{nist}
$$\psi'(x)=\int_0^{\infty}\frac{t}{1-e^{-t}}e^{-xt}\dt\ \ \ \mbox{and}\ \ \ \frac{1}{x^2}=\int_0^{\infty}te^{-xt}\dt,$$
we obtain that
\begin{equation}\label{h'}h'(\nu)=\psi'\left(\nu+\frac{3}{2}\right)-\psi'(\nu+2)-\frac{1}{2(\nu+1)^2}=
\frac{1}{2}\int_0^{\infty}\frac{te^{-(\nu+1)t}}{1-e^{-t}}\left(2e^{-\frac{1}{2}t}-e^{-t}-1\right)\dt<0\end{equation}
for all $\nu>-1.$ We note that by using the series representation \cite[p. 139]{nist}
$$\psi(x+1)=-\gamma+\sum_{n=1}^{\infty}\left(\frac{1}{n}-\frac{1}{n+x}\right),$$
where $\gamma$ is the Euler constant, it follows that
$$2h(\nu)=\frac{1}{\nu+1}-\sum_{n=1}^{\infty}\frac{1}{(n+\nu+1)(n+\nu+\frac{1}{2})},$$
which shows that $h(\nu)\to0,$ as $\nu\to\infty.$ Consequently, $h(\nu)>0$ if $\nu\in(-1,\infty).$ Thus, the function $g$ is increasing, and $g(\nu)>g\left(-\frac{1}{2}\right)=1$ if $\nu>-\frac{1}{2}.$
\end{proof}

\begin{remark}
{\em We mention that actually the right-hand side of \eqref{gammaineq} can be rewritten as the Tur\'an type inequality
$$\frac{2}{\pi}<\frac{\Gamma^2\left(\nu+\frac{3}{2}\right)}{\Gamma(\nu+1)\Gamma(\nu+2)}.$$ Moreover, we note that by using the recurrence relation $\Gamma(x+1)=x\Gamma(x),$ the inequality \eqref{gammaineq} can be rewritten in the form
$$\frac{2}{\sqrt{\pi}}\frac{\nu+1}{\nu+\frac{1}{2}}>\frac{\Gamma\left(\nu+\frac{1}{2}\right)}{\Gamma(\nu+1)}>
\sqrt{\frac{2}{\pi}}\frac{\sqrt{\nu+1}}{\nu+\frac{1}{2}},$$
which is valid for $\nu>-\frac{1}{2}.$ As far as we know the above inequality is new. Lower and upper bounds for the quotient $\Gamma\left(x+\frac{1}{2}\right)/\Gamma(x+1)$ have been established by many authors, we refer to the survey paper \cite{qi} for more details.

It is also important to note that according to \eqref{h'} the function $h$ is actually completely monotonic. }
\end{remark}

\subsection*{Acknowledgement} The research of \'A. Baricz was supported by the J\'anos Bolyai Research Scholarship of the Hungarian Academy of Sciences. The authors are very grateful to the referee for his/her appropriate and constructive suggestions and for his/her proposed corrections to improve the paper.

\end{document}